\documentclass[a4paper,12pt,oneside,reqno ]{amsart}

\usepackage[T1]{fontenc}
\usepackage[utf8]{inputenc}
\usepackage{lmodern}
\usepackage[english]{babel}

\usepackage[%
left=2.5cm,       
right=2.5cm,      
top=3.5cm,        
bottom=3.5cm,     
heightrounded,    
bindingoffset=0mm 
]{geometry}
\usepackage{amsrefs}
\usepackage{amsmath}
\usepackage{amssymb}
\usepackage{mathtools}
\usepackage{mathrsfs}
\usepackage[colorlinks=true,linkcolor=blue,urlcolor=blue,citecolor=blue]{hyperref}

\numberwithin{equation}{section}

\usepackage{hyperref} 

\usepackage{tikz}
\usepackage{graphicx}
\usepackage{xcolor}
\usepackage[font=small]{caption}

\usepackage{bookmark}
\usepackage{enumitem}

\theoremstyle{plain}
\newtheorem{theorem}{Theorem}[section]

\newtheorem{proposition}[theorem]{Proposition}
\newtheorem{corollary}[theorem]{Corollary}

\theoremstyle{definition}
\newtheorem{definition}[theorem]{Definition}
\newtheorem{example}[theorem]{Example}
\newtheorem{remark}[theorem]{Remark}
\newtheorem{open.problem}[theorem]{Open Problem}

\newcommand{\N}{\mathbb{N}}
\newcommand{\R}{\mathbb{R}}

\title[Brunn-Minkowski inequality for Schr\"odinger operators]{A Brunn-Minkowski inequality for Schr\"odinger operators with Kato class potentials}
\author[A.\ Carbotti]{Alessandro Carbotti}
\address{Dipartimento di Matematica
	e Fisica ``E. De Giorgi'', Universit\`a del Salento,
	Via Per Arnesano, 73100 Lecce, Italy.}
\email{alessandro.carbotti@unisalento.it}

\date{\today}  



\begin{document}
	\begin{abstract}
		In this paper we prove a Brunn-Minkowski inequality for the first Dirichlet eigenvalue of a Schr\"odinger type operator $\mathcal{H}_V:=-\operatorname{div}(A\nabla)+V$, where $V$ is convex and Kato decomposable, using the trace class property of the generated semigroup. As a consequence, we obtain the log-concavity of the ground state using the ultracontractivity of the semigroup, and also the strong log-concavity under additional assumptions on $\Omega$ and $V$.
	\end{abstract}
	
	\maketitle
	
	\tableofcontents
	
	\section{Introduction}
	\label{sec:intro}
	The Brunn-Minkowski inequality is a classical topic in Geometric Analysis. 
	In its simplest form, it asserts that for any nonempty Borel sets $\Omega_0,\Omega_1 \subset \mathbb{R}^N$,  set for any $r \in [0,1]$ the convex Minkowski sum $\Omega_r:=(1-r)\Omega_0+r\Omega_1$, one has that
	\begin{equation} \label{eq:BM-classical}
		|\Omega_r|^{1/N} \geq (1-r)\,|\Omega_0|^{1/N} + r\,|\Omega_1|^{1/N},
	\end{equation}
	where $|\cdot|$ denotes the Lebesgue measure. In other words, for every nonempty Borel set $\Omega$, the function $\Omega\mapsto |\Omega|^{1/N}$ is concave.  
	Inequality \eqref{eq:BM-classical} has far-reaching implications, including the isoperimetric inequality, concentration phenomena and several other functional inequalities. See \cites{BobkovLedoux:1997, GardnerSurvey:2002} for several applications.

	Inequality \eqref{eq:BM-classical} is a particular case of the Prékopa-Leindler inequality, established in \cites{prekopa,Leindler:1972} and generalized in higher dimension in \cites{BorellPL, Prekopa71}, which can be stated as follows.
	\begin{theorem}
    \label{th:Pltheorem}
		Let $\Omega\subseteq\R^N$ be a measurable and convex set and $f,g,h:\Omega\rightarrow (0,+\infty)$ measurable functions such that
		\begin{equation}
        \label{eq:fgh}
		h((1-r)x+ry)\ge f(x)^{1-r}g(y)^r
		\end{equation}
		for every $x,y\in \Omega$ and every $r\in[0,1]$. Then
		$$
		\int_{\Omega}h(x)dx\ge\left(\int_{\Omega}f(x)dx\right)^{1-r}\left(\int_{\Omega}g(x)dx\right)^r
		$$
	\end{theorem}

		Theorem \ref{th:Pltheorem} has been crucial in the study of concavity properties of solutions to elliptic and parabolic equations. In particular, condition \eqref{eq:fgh} when $f=g=h$ is the definition of \textit{log-concave} functions. In recent years, numerous extensions and refinements of \eqref{eq:BM-classical} and Theorem \ref{th:Pltheorem} have been proposed for other geometric functionals arising in Calculus of variations and for other notions of concavity. See \cites{Colesanti05, GardnerSurvey:2002} and the references therein. Apart from the euclidean setting, in the Gauss space the following inequality 
	\begin{equation}
    \label{eq:BMgaussvol}
	\gamma(\Omega_r)\ge \gamma(\Omega_0)^{1-r}\gamma(\Omega_1)^r,
	\end{equation}
	where $\gamma$ denotes the standard Gaussian measure $\gamma:=\frac{e^{-|\cdot|^2/2}}{(2\pi)^{N/2}}\mathcal{L}^N$, 
	expresses that the Gaussian volume enjoys a \textit{weak} Brunn-Minkowski inequality which emphasizes the log-concavity of the measure $\gamma$. With a little abuse of notation, in the sequel we will refer to $\gamma$ to indicate both the Gaussian measure and its density with respect to the Lebesgue measure.
    
    Inequality \eqref{eq:BMgaussvol} holds true as an immediate application of Theorem \ref{th:Pltheorem} and has been proved with other methods by Borell in \cite{borell}. An improvement of inequality \eqref{eq:BMgaussvol} has been given in \cite{EskMos21}, where in the class of symmetric sets with respect to the origin the authors prove \eqref{eq:BM-classical} with the Lebesgue measure replaced by the Gaussian one. Moving to the spectral framework, let $w\in\left\{1,\gamma\right\}$, and let $\lambda_{1,w}(\Omega)$ be the least real number $\lambda$ such that the problem
    \begin{equation}
    \begin{cases}
        -\operatorname{div}(w\nabla v)=\lambda wv\quad\text{in}\quad\Omega \\
        v=0\quad\text{in}\quad\partial\Omega
    \end{cases}
    \end{equation}
    admits a nontrivial solution $u$. A Brunn-Minkowski inequality for $\lambda_{1,w}(\Omega)$ goes as follows
    	\begin{equation}
	\label{eq:eigenBM}
	\lambda_{1,w}(\Omega_r)\le (1-r)\lambda_{1,w}(\Omega_0)+r\lambda_{1,w}(\Omega_1).
	\end{equation}
 In the case $w=1$ the convexity of the first eigenvalue of the Dirichlet Laplacian \eqref{eq:eigenBM} is equivalent to the concavity of the function $\Omega\mapsto\lambda^{-1/2}_{1}(\Omega)$ proved in \cite{BrLi76}. If $w=\gamma$ inequality \eqref{eq:eigenBM} deals with the convexity of the first Dirichlet eigenvalue 
    of the Ornstein-Uhlenbeck operator $-\Delta_\gamma:=-\Delta+x\cdot\nabla$ recently proved in \cite[Theorem 1.2]{CFLS}, where the authors also address the case of equality in \eqref{eq:eigenBM}. We want to notice that $\lambda_{1,\gamma}(\Omega)$ also enjoys a Faber-Krahn inequality, see \cites{BeChFe, CCLP24a}. The Brunn-Minkowski inequality for $\lambda_{1,\gamma}(\Omega)$ can also be restated for the first Dirichlet eigenvalue of the Schr\"odinger operator $\mathcal{H}:=-\Delta+\frac{|\cdot|^2}{4}-\frac N2$ using the unitary transformation \eqref{eq:unitarytransform} defined in the following Section \ref{sec:proof}, which induces an isospectrality between $-\Delta_\gamma$ and $\mathcal{H}$. This connection highlights the relevance of Schr\"odinger-type operators, whose spectral properties play a central role both in modern analysis and mathematical physics. A notable example illustrating this connection is the proof of the fundamental gap conjecture by Andrews and Clutterbuck in \cite{AndClu}.
	
	In the spirit of exploring geometric properties of the spectrum of the Schr\"odinger operator, we establish in this article a Brunn-Minkowski-type inequality for the first Dirichlet eigenvalue of a Schr\"odinger type operator: Namely, we consider the following boundary value problem
		\begin{equation}
		\begin{cases}
			\mathcal{H}_Vu=\lambda_{1,V}(\Omega) u\quad\text{in}\quad\Omega \\
			u=0\quad\text{in}\quad\partial\Omega
		\end{cases}
	\end{equation}
	being
	$$
	\mathcal{H}_V:=-\operatorname{div}(A\nabla)+V
	$$
	where $A$ is a constant symmetric matrix such that $A\xi\cdot\xi\ge a_1|\xi|^2$ for every $\xi\in\R^N$ and for some $a_1>0$. Denoting $k^+:=\max\{k,0\}$ and $k^-:=\max\{-k,0\}$ for each $k\in\R$, we also assume that the potential $V:\Omega\rightarrow\R$ satisfies the following assumptions: 
	\begin{enumerate}[label=\textbf{A.\arabic*}]
		\item\label{As1} $V$ is convex in $\Omega$.
		\item\label{As2} $V$ is \textit{Kato decomposable}, namely $V^+\in L^1_{\rm loc}(\Omega)$ and $V^-\in\mathcal{K}(\Omega)$, where $\mathcal{K}(\Omega)$ denotes the Kato class. See Definition \ref{def:katoclass} below.
		\item\label{As3}$e^{-tV}\in L^1(\Omega)$ for every $t>0$.
	\end{enumerate}
    With this notation in hand, we can state our main result
		\begin{theorem}
		\label{teo:maintheorem}
		Let $\Omega_0, \Omega_1$ be non-empty convex sets in $\R^N$. For every $r\in[0,1]$ we set
		$$
		\Omega_r:=(1-r)\Omega_0+r\Omega_1.
		$$
		Then, if assumptions \ref{As1} is satisfied in $\displaystyle\bigcup_{r\in[0,1]}\Omega_r$ and assumptions \ref{As2}, \ref{As3} are satisfied in $\Omega_r$ for every $r\in[0,1]$ it holds that
		$$
		\lambda_{1,V}(\Omega_r)\le (1-r)\lambda_{1,V}(\Omega_0)+r\lambda_{1,V}(\Omega_1).
		$$
	\end{theorem}

    A relevant consequence is the following Theorem which asserts the log-concavity of the ground state

    \begin{theorem}
\label{th:logconcdireig}
	Let $\Omega\subset\R^N$ be an open and connected convex set and let $\psi_{1,V}$ be the first Dirichlet eigenfunction of $\mathcal{H}_V$ in $\Omega$. Then $\psi_{1,V}$ is strictly positive in $\Omega$ and log-concave in $\Omega$.
\end{theorem}

	Some comments are in order. Our approach relies on the pioneering papers by Brascamp and Lieb \cites{Lieb, BrLi76}. The assumption of taking constant second order coefficients is sharp in order to ensure the log-concavity of the heat kernel. See \cite[Theorem 1.2, Proposition 1.3]{Koles01} and Section \ref{sec:proof}. Concerning the potential term $V$ assumptions \ref{As1}, \ref{As2} and \ref{As3} are classical. More specifically, Assumption \ref{As1} ensures that the heat kernel of $\mathcal{H}_V$ is log-concave in the spatial variables; Assumption \ref{As2} guarantees that the heat kernel enjoys upper Gaussian estimates; finally, Assumption \ref{As3} implies that the semigroup generated in $L^2(\Omega)$ by $-\mathcal{H}_V$ with Dirichlet boundary conditions in $\Omega$ is trace class and so the spectrum of $\mathcal{H}_V$ in $L^2(\Omega)$ is discrete. We refer to \cite{MazShu} for other compactness criteria ensuring the discreteness of the spectrum. Moreover, Assumptions \ref{As1} and \ref{As3} forces the potential $V$ to be bounded below, ensuring that the associated quadratic form is also bounded below. We observe that one can also address the problem of optimizing the first eigenvalue $\lambda_{1,V}(\Omega)$ with respect to the potential $V$ under the constraint given by the $L^1(\Omega)$ norm of $e^{-tV}$ for some $t>0$, as recently done in \cite{Frank25}. It is noteworthy that under the above assumptions on $V$, the first eigenvalue is allowed to be negative, a situation that can clearly occur in the case of sign changing or negative potentials. Indeed, the boundedness from below of the potential forces the spectrum to lie in a right-half line. We notice that also in these cases one can obtain a positive spectrum providing the existence of nonnegative solutions to the Poisson equation for $\mathcal{H}_V$ with measure datum or pointwise lower bounds on the potential in terms of some barrier functions as done in \cites{BuPoOr22, BraFraRuf18}. For a more complete treatment on Schr\"odinger semigroups we refer the interested reader to the fundamental work \cite{Simon82}.

	The rest of the paper is organized as follows. 
	In Section~\ref{sec:preliminaries} we introduce the notation, the geometric setting, and the functional framework required for our analysis.
	In Section~\ref{sec:proof} we prove Theorem \ref{teo:maintheorem} and apply it to prove Theorem \ref{th:logconcdireig} exploiting the ultracontractivity of the semigroup generated by $-\mathcal{H}_V$ with homogeneous Dirichlet boundary conditions. We conclude the paper by proving that if both $\Omega$ and $V$ satisfy stronger regularity and convexity assumptions the ground state $\psi_{1,V}$ enjoys a strong log-concavity.

	\paragraph*{\bf Acknowledgements}
    The author warmly thanks Luigi Negro and Diego Pallara for many useful discussions and suggestions. 
    
	A.C. is member of GNAMPA of the Italian Istituto Nazionale di Alta Matematica (INdAM) and has been partially supported
	by the INdAM - GNAMPA 2026 Project ``Analisi variazionale per operatori locali e nonlocali
	possibilmente singolari o degeneri''. 
	\vspace{.5cm}
	\section{Preliminary Results}
	\label{sec:preliminaries} 
	
	In the sequel we assume that $D\subseteq\R^N$ is a nonempty open set and that $\Omega\subset\R^N$ is a nonempty connected convex set in $\R^N$. We will also denote with $L^\Omega$ the realization of a second order operator endowed with homogeneous Dirichlet boundary conditions in $\Omega$.

    To properly define the Schr\"odinger operator and the associated semigroup we now define the Kato class, introduced by Kato in \cite{Kato72} to manage singular perturbations of the laplacian.
	
		\begin{definition}
		\label{def:katoclass}
	Let $N\ge 2$. We say that the function $W$ belongs to the \textit{Kato class} $\mathcal{K}(D)$ if 
		$$
		\lim_{r\to0^+}\sup_{x\in D}\int_{B_r(x)\cap D}|W(y)|\mathbb{G}_N(x-y)dy=0,
		$$
		where $\mathbb{G}_N$ denotes the Green function for the Laplacian in $\R^N$.
        
        Otherwise, if $N=1$ we say $W\in\mathcal{K}(D)$ if 
        $$
        \sup_{x\in D}\int_{B_1(x)\cap D}|W(y)|dy<\infty.
        $$
	\end{definition}

    By using H\"older inequality one can easily check that $L^p(D)\subset\mathcal{K}(D)$ for every $p>N/2$ if $N\ge 2$, and for every $p\ge 1$ if $N=1$.

     \subsection{Trace class semigroups}

We recall here some well known facts about trace class operators. We refer the interested reader e.g. to \cite[Chapt. 5, Sec. 6]{Davies07}. Let $H$ be an infinite dimension separable Hilbert space and let $\mathcal{B}(H)$ denote the algebra of bounded linear operators on $H$. For any operator $T \in \mathcal{B}(H)$, we define its absolute value as $|T|:=\sqrt{T^*T}$.

\begin{definition}[Trace Class Operator]
An operator $T \in \mathcal{B}(H)$ is said to be \textit{trace class} (or strictly nuclear) if, for some orthonormal basis $\{e_k\}_{k=1}^\infty$ of $H$, the following sum converges:
$$\sum_{k=1}^\infty (|T| e_k, e_k)_H < \infty.$$
\end{definition}

If $T$ is trace class, the value of the sum does not depend on the choice of the orthonormal basis. We denote the space of all trace class operators by $\mathcal{S}_1(H)$. For any $T\in\mathcal{S}_1(H)$, the \textit{trace} of $T$ is defined as:
$$\operatorname{Tr}(T) = \sum_{k=1}^\infty (T e_k, e_k)_H.$$
This series converges absolutely and is independent of the chosen basis. Trace class operators can be approximated in the $\mathcal{B}(H)$ norm by finite rank operators, so they are compact operators.

Now, let $(T(t))_{t \ge 0}$ be a selfadjoint strongly continuous $C_0$-semigroup of bounded linear operators on $H$, and let $-L$ with domain $D(L)$ be its infinitesimal generator. When the generator is defined we denote by $(e^{-tL})_{t>0}$ the associated semigroup.

If $-L$ is the infinitesimal generator of a self-adjoint semigroup
$(e^{-tL})_{t>0}$ and $L$ has compact resolvent, then its spectrum
consists of isolated eigenvalues of finite multiplicity, which we denote by
$\{\lambda_k\}_{k\in\mathbb N}$.

To link the trace of the semigroup $T(t)$ directly to the eigenvalues of its generator $-L$, we employ the Spectral Mapping Theorem. See e.g. \cite[Chapt. IV]{EngelNagel00}. 

\begin{theorem}[Spectral Mapping Theorem for Point Spectrum]
\label{th:specmapth}
For a strongly continuous semigroup $(e^{-tL})_{t>0}$ generated by $-L$, the point spectrum obeys the following relation:
\begin{equation}
\label{eq:specmaptheo}
e^{t \sigma_p(-L)} \subseteq \sigma_p(e^{-tL})\setminus \{0\}.
\end{equation}
If $(e^{-tL})_{t>0}$ is also compact, the equality holds in \eqref{eq:specmaptheo}.
\end{theorem}

A useful criterion to satisfy the trace class property is given by the following Proposition.

\begin{proposition}[Trace of a self-adjoint semigroup]
\label{prop:trace-semigroup}
Let $L$ be a self-adjoint operator on a separable Hilbert space $H$. Assume that $L$ has compact resolvent and that $L-\omega I$ is accretive for some $\omega\in\R$. Let
$\{\lambda_k\}_{k\in\mathbb N}$ be the sequence of eigenvalues, repeated according
to their multiplicity. Then, for every $t>0$,
\begin{equation}
\label{eq:trcleq}
e^{-tL}\in \mathcal S_1(H)
\quad\Longleftrightarrow\quad
\sum_{k=1}^{\infty}e^{-t\lambda_k}<\infty.
\end{equation}
Whenever the equivalent condition in \eqref{eq:trcleq} hold, the trace of the semigroup $(e^{-tL})_{t>0}$ is given by
$$
\operatorname{Tr}(e^{-tL})
   =\sum_{k=1}^{\infty}e^{-t\lambda_k}.
$$
In particular, $(e^{-tL})_{t>0}$ is a trace class semigroup if and
only if the above series converges for every $t>0$.
\end{proposition}

\begin{proof}
Using Theorem \ref{th:specmapth} for self-adjoint operators and the characterization of positive
trace class operators; see
\cite[Lemma 5.6.2 and Section 7.2]{Davies07}, the claim plainly follows.
\end{proof}

\begin{remark}
    Let $H=L^2(D)$ and let $\mathfrak{a}:D(\mathfrak{a})\times D(\mathfrak{a})\rightarrow \mathbb{C}$ be a symmetric, closed, bounded from below and densely defined sesquilinear form with $D(\mathfrak{a})\subset H$, and $L$ the operator associated to $\mathfrak{a}$ by the formula
    $$
    (Lu,v)_H:=\mathfrak{a}(u,v),
    $$
    for every $u\in D(L)$, $v\in D(\mathfrak{a})$, where $(\cdot,\cdot)_H$ denotes the scalar product in $H$. By the standard Theory (see e.g. \cite[Proposition 1.51]{Ou05}) the operator $-L$ generates a selfadjoint strongly continuous semigroup in $H$. Moreover, if the embedding
$$
D(\mathfrak a)\hookrightarrow H
$$
is compact, then $L$ has compact resolvent and its spectrum consists of
isolated eigenvalues of finite multiplicity. Hence there exists an
orthonormal basis $(\psi_k)_{k\in\mathbb N}$ of eigenfunctions and
$$e^{-tL}f=
\sum_{k=1}^{\infty}
e^{-t\lambda_k}(f,\psi_k)_H\psi_k
\quad\text{in}\quad H.
$$
If, in addition, $(e^{-tL})_{t>0}$ admits an integral kernel
$p_L(t,x,y)$ and is trace class, then
$$
p_L(t,x,y)=\sum_{k=1}^{\infty}
e^{-t\lambda_k}\psi_k(x)\psi_k(y)
$$
in $L^2(D\times D)$.
If, in addition, $(e^{-tL})_{t>0}$ admits a continuous integral kernel
$p_L(t,x,y)$, Mercer's formula
\cite[Proposition 5.6.9]{Davies07} gives
$$
\operatorname{Tr}(e^{-tL})
=
\int_Dp_L(t,x,x)\,dx
=
\sum_{k=1}^{\infty}e^{-t\lambda_k}.
$$
Consequently, $(e^{-tL})_{t>0}$ is a
trace-class semigroup if and only if
$$
Z_D(t):=\int_Dp_L(t,x,x)\,dx<\infty
$$
for every $t>0$.
\end{remark}

Another class of semigroups enjoying good spectral properties is the class of  \textit{irreducible semigroups}. We state this notion in the setting of $L^2$-spaces referring to \cite[Chapt. 2, Sect. 2]{Ou05} for a more complete overview.

\begin{definition}
    Let $\mu$ be a $\sigma$-finite measure and $(T(t))_{t\ge 0}$ be a positive $C_0$ semigroup in $L^2(D,\mu)$. Then, $T(t)$ is irreducible if, for every nonzero $f\in L^2(D)$, $f\ge 0$, it holds that $T(t)f(x)>0$ for $\mu$-almost every $x\in D$.
\end{definition}

\subsection{The Schr\"odinger semigroup in $L^2(\Omega)$ with Kato class potential}

In this Subsection we recall some fundamental results of the semigroup generated by $-\mathcal{H}_V$ in $L^2(\Omega)$ with Dirichlet conditions. To do this, we need some tools from stochastic processes in order to define the Dirichlet-Schr\"odinger semigroup in $L^2(\Omega)$. We refer for more details to (see \cite[Chapter 1, Section 1.6]{ChuZha95})

\begin{definition}
Let $(X_s)_{s\geq0}$ be the diffusion process generated by
$\operatorname{div}(A\nabla)$, and let
$$
\tau_\Omega:=\inf\{s>0:X_s\notin\Omega\}
$$
be its first exit time from $\Omega$. By the Feynman-Kac formula for
the Dirichlet Schrödinger semigroup, for every nonnegative
$f\in L^2(\Omega)$ we have
\begin{equation}
\label{eq:feynmakac}
(e^{-tH_V^\Omega}f)(x)
=
\mathbb E_x\left[
\exp\left(-\int_0^t V(X_s)\,ds\right)
f(X_t)\mathbf 1_{\{t<\tau_\Omega\}}
\right].
\end{equation}
\end{definition}

Now, let $p_A^\Omega$ denote the Dirichlet heat kernel associated with
$\operatorname{div}(A\nabla)$ in $\Omega$. 
Since by Assumptions \ref{As1} and \ref{As3} one has $V\geq c_{V,\Omega}$ in $\Omega$ for some $c_{V,\Omega}\in\R$, it follows
that
$$
0\leq e^{-tH_V^\Omega}f
\leq e^{-c_{V,\Omega}t}e^{t\operatorname{div}(A\nabla)^\Omega}f.
$$
Consequently,
$$
0\leq p_V^\Omega(t,x,y)
\leq e^{-c_{V,\Omega}t}p_A^\Omega(t,x,y)
\leq e^{-c_{V,\Omega}t}p_A(t,x,y),
$$
where
\begin{equation}
\label{eq:gaussAkernel}
p_A(t,x,y)
=
\frac{1}{\sqrt{\det A}}\frac{1}{(4\pi t)^{N/2}}
\exp\left(
-\frac{|A^{-1/2}(x-y)|^2}{4t}
\right).
\end{equation}
In particular, for every $t>0$ the operator $(e^{-tH_V^\Omega})_{t>0}:L^2(\Omega)\rightarrow L^\infty(\Omega)$ is bounded. By self-adjointness, immediately follows that for every $t>0$ the operator $(e^{-tH_V^\Omega})_{t>0}:L^1(\Omega)\rightarrow L^\infty(\Omega)$ is bounded and the semigroup enjoys the following ultracontractivity estimate
\begin{equation}
\label{eq:ultracontractivity}
\left\|e^{-t\mathcal{H}^\Omega_V}\right\|_{\mathcal{B}(L^1(\Omega),L^\infty(\Omega))}=\sup_{x,y\in\Omega} |p^\Omega_V(t,x,y)|\le C\frac{e^{-c_{V,\Omega} t}}{\min\{1,t^{N/2}\}}\quad\,t>0.
\end{equation}
Moreover, the semigroup $(e^{-t\mathcal{H}^\Omega_V})_{t>0}$ is also trace class by the Golden-Thompson-Symanzik estimate, (see e.g. \cite[Theorem 9.2]{Simon79})
\begin{equation}
\label{eq:gtsestimate}
\operatorname{Tr}(e^{-t\mathcal{H}^\Omega_V})\le \operatorname{Tr}(e^{t\operatorname{div}(A\nabla)^\Omega}e^{-tV})\le \frac{C}{\min\{1,t^{N/2}\}}\int_\Omega e^{-tV(x)}dx,\quad t>0,
\end{equation}
where the right-hand side is finite thanks to Assumption \ref{As3}.

\vspace{0.2 cm}

\subsection{The spectrum of $\mathcal{H}^\Omega_V$ in $L^2(\Omega)$}

	 Now, we consider the bilinear form
     $$
	 {\mathcal E_V}(u,v) := \int_{\Omega} A\nabla u \cdot \nabla v\, dx+\int_{\Omega}Vuvdx,
	 $$
	 and we set $\mathcal{Q}_V(u):=\mathcal{E}_V(u,u)$.
	  Using the assumptions on $A$ and $V$ we have that 
	  \begin{equation}
			\mathcal{Q}_V(u)\ge a_1\int_{\Omega}|\nabla u|^2dx+c_{V,\Omega}\int_\Omega u^2dx,
	\end{equation}
with $c_{V,\Omega}:=\inf_{x\in\Omega} V(x)>-\infty$. Now, if we assume that the set $\Omega$ supports the Poincarè inequality with positive spectrum
$$
\int_\Omega \varphi^2dx\le\frac{1}{\lambda_{1,0}(\Omega)}\int_\Omega|\nabla\varphi|^2dx,
$$
for every $\varphi\in H^1_0(\Omega)$, being $\lambda_{1,0}(\Omega)>0$ the first eigenvalue of $-\Delta^\Omega$, we can write
\begin{equation*}
\mathcal{Q}_V(u)\ge a_1\lambda_{1,0}(\Omega)\int_\Omega u^2dx+c_{V,\Omega}\int_{\Omega}u^2dx
\end{equation*}
which implies that the form $\mathcal{Q}_V$ is bounded below. In particular, if $a_1\lambda_{1,0}(\Omega)+c_{V,\Omega}>0$, $\mathcal{Q}_V$ is coercive.

	 By the trace class condition \eqref{eq:gtsestimate} satisfied by $(e^{-t\mathcal{H}^\Omega_V})_{t>0}$ we have that the embedding
	\begin{equation}
		\label{eq:compactembedding}
		D(\mathcal{Q}_V):=H^1_0(\Omega)\cap L_V^2(\Omega)\hookrightarrow L^2(\Omega),
	\end{equation}
	being $L^2_V(\Omega):=\left\{w\in L^2(\Omega),\,\sqrt{V_+}w\in L^2(\Omega)\right\}$, is compact and the spectrum is discrete. Therefore, by the boundedness from below of $\mathcal{Q}_V$ the eigenvalues of $\mathcal{H}_V$ define an increasing and positively diverging sequence $$\lambda_{1,V}(\Omega)\le\lambda_{2,V}(\Omega)\le\ldots\lambda_{k,V}(\Omega)\le\ldots,$$ and we can write $\lambda_{1,V}(\Omega)$ in terms of the Rayleigh quotient
	$$
	\lambda_{1,V}(\Omega)=\inf_{u\in D(\mathcal{Q}_V)\setminus\{0\}}\frac{\mathcal{Q}_V(u)}{\left\|u\right\|^2_{L^2(\Omega)}}.
	$$
 We refer to \cite{MazShu} for a complete characterization of the spectrum of $\mathcal{H}_V$.

 We conclude this Section with the following results that allow to obtain further information on the spectral properties of the Schr\"odinger semigroup 

 \begin{proposition}
\label{prop:irrschdirsem}
    Let $D\subset\R^N$ an open connected set and $(e^{-t\mathcal{H}^D_V})_{t>0}$ the semigroup generated by $-\mathcal{H}_V$ with Dirichlet condition in $L^2(D)$. Then, $(e^{-t\mathcal{H}^D_V})_{t>0}$ is positive and irreducible.
\end{proposition}
\begin{proof}
    The positivity is a consequence of the first inequality in \eqref{eq:gaussAkernel}. The irreducibility is proved in \cite[Chapter XIII.12]{ReedSimonIV}.
\end{proof}

\begin{remark}
    Since $(e^{-t\mathcal{H}_V^D})_{t>0}$ is a positive, selfadjoint and irreducible semigroup in $L^2(D)$ it follows that $$\omega_0=s(-\mathcal{H}^D_V):=
	\sup \{\operatorname{Re} \lambda : \lambda \in \sigma(-\mathcal{H}^D_V) \}=-\lambda_{1,V}(D)$$ where $\lambda_{1,V}(D)$ is simple by Krein-Rutman Theorem. See, for instance, \cite[Proposition~3.4, Chapter~VI]{EngelNagel}.
\end{remark}

\section{Log-concavity of the heat kernel}
\label{sec:proof}

In this Section we give the proof of Theorem \ref{teo:maintheorem}. As already said in Section \ref{sec:intro}, the constancy assumption on the second order coefficients is sharp in order to ensure the log concavity of the heat kernel as given by the following Theorem (See \cite[Theorem 1.2, Proposition 1.3]{Koles01}).

\begin{theorem}
\label{th:kolesnikov}
    Let $Q(x)=(q_{ij}(x))_{i,j=1,\ldots N}$ and $\beta(x)=(\beta_h(x))_{h=1,\ldots,N}$ be such that $q_{ij}$, $\beta_h\in C^{2,\delta}(B_R)$ for some $\delta>0$ and for every $R>0$, $i,j,h=1,\ldots, N$. Let $L=\operatorname{Tr}(Q(\cdot)D^2)+\beta(\cdot)\cdot\nabla$ and let $(e^{tL})_{t>0}$ be the associated semigroup. Then,  $(e^{tL})_{t>0}$ preserves log-concavity for every $t>0$ if, and only if, $Q$ is constant and $\beta$ is affine.
\end{theorem}

Now we are ready to prove our main result

\begin{proof}[Proof of Theorem \ref{teo:maintheorem}]
	
	We follow the computations done in \cites{Lieb}. Let $p_{V,r}(t,x,y)$ be the fundamental solution of 
	\begin{equation*}
		\begin{cases*}
			U_t+\mathcal{H}_VU=0\quad\text{in}\quad(0,+\infty)\times\Omega_r\times \Omega_r \\
			U(0,x,y)=\delta(x-y)\quad\text{in}\quad \Omega_r\times \Omega_r \\
			U(t,x,y)=0\quad\text{in}\quad(0,+\infty)\times\partial \Omega_r\times \Omega_r\\
			U(t,x,y)=0\quad\text{in}\quad (0,+\infty)\times\Omega_r^c\times \Omega_r\cup (0,+\infty)\times\Omega_r\times \Omega^c_r,
		\end{cases*}
	\end{equation*}
	where $\Omega_r$ is the convex sum of $\Omega_0$ and $\Omega_1$ and $\mathcal{H}_{V,r}:=\mathcal{H}_V^{\Omega_r}$ denotes the realization of $\mathcal{H}_V$ with Dirichlet boundary conditions on $\Omega_r$.  
    With this notation in force the heat kernel of $-\mathcal{H}_{V,r}$ is given by the Trotter perturbation formula, (see e.g. \cite[Corollary 5.8]{EngelNagel00})
    \begin{equation}
    \label{eq:Trotterformula}
	\begin{split}
	p_{V,r}(t,x,y)&=\lim_{n\to\infty}p_{V,r,n}(t,x,y) \\
    &:=\lim_{n\to\infty}\int_{\R^{N(n-1)}}\left(\prod_{k=1}^n p_{A}\left(\frac tn,x_k, x_{k-1}\right)e^{-\frac tnV(x_k)}\chi_{\Omega_r}(x_k)\right)dx_1\ldots dx_{n-1},
	\end{split}
    \end{equation}
	where $p_A$ is given as in \eqref{eq:gaussAkernel}, $x_0:=x$ and $x_n:=y$.
	
By the log-concavity of the integrand and Theorem \ref{th:Pltheorem}, the function
$r\mapsto Z_n(r,t)$ is log-concave for every $t>0$ and every $n\in\N$, which means that
	$$
	Z_n(r,t)\ge Z_n(0,t)^{1-r}Z_n(1,t)^{r},
	$$
	for every $r\in[0,1]$, $t>0$ and $n\in\N$. By \eqref{eq:Trotterformula} it follows that also the limit trace $Z(r,t)=\int_{\Omega_r}p_{V,r}(t,x,x)dx$ is log-concave in $r$ for every $t>0$.
    Since the semigroup $(e^{-t\mathcal{H}_{V,r}})_{t>0}$ is trace class, we have that
	$$
	Z(r,t)=\sum_{k=1}^\infty e^{-t\lambda_{k,V}(\Omega_r)}<\infty,
	$$
    for every $t>0$.
	 In particular
	$$
	\lambda_{1,V}(\Omega_r)=-\lim_{t\to\infty}\frac{\log Z(r,t)}{t}
	$$
	and $\lambda_{1}(r)$ is convex in $r$ since it is the pointwise limit of convex functions.
\end{proof}

\begin{example}[Kolmogorov Operators]
\label{example:Kolmogorov}
		Let $A$, $B\in\R^{N,N}$ be constant and symmetric matrices such that $AB=BA$, with $A$ positive definite and $b_0\in\R^N$. Consider the Kolmogorov operator
		\begin{equation}
        \label{eq:kolmexa}
		\mathcal{L}=\operatorname{div}(A\nabla)-b\cdot\nabla,
		\end{equation}
		being $b(x):=Bx+b_0$, endowed with Dirichlet boundary condition on $\partial\Omega$ and we refer to it by $\mathcal{L}^\Omega$. We notice that the semigroup $(e^{t\mathcal{L}^\Omega})_{t>0}$ satisfies the hypotheses of Theorem \ref{th:kolesnikov}. Furthermore, the transformation
	\begin{equation}
		\label{eq:unitarytransform}
		\begin{split}
			U_\varphi:L^2(\Omega,e^{-2\varphi}\mathcal{L}^N)\rightarrow L^2(\Omega) \\
			f\mapsto e^{-\varphi }f,
		\end{split}
	\end{equation}
	where 
    \begin{equation}
    \label{eq:quadratic}
    \varphi(x):=\frac{A^{-1}Bx\cdot x}{4}+\frac{A^{-1}b_0\cdot x}{2},
    \end{equation}
    defines an isometry between $L^2(\Omega,e^{-2\varphi}\mathcal{L}^N)$ and $L^2(\Omega)$. 
	
	Set $V_{\varphi}:=A\nabla\varphi\cdot\nabla\varphi-\operatorname{div}(A\nabla\varphi)=\frac {b\cdot A^{-1}b}{4}-\frac{\operatorname{Tr(B)}}{2}$. It is easy to check that $\mathcal{H}_{V_\varphi}(e^{-\varphi})=0$ in $\R^N$, and so the isometry $U_\varphi$ yields $$-\mathcal{L}^\Omega=U_{\varphi}^{-1}\mathcal{H}^\Omega_{V_\varphi}U_\varphi.$$ Since $A$ and $V_\varphi$ fulfill the hypotheses given in Section \ref{sec:preliminaries} and $\mathcal{L}^\Omega$ has the same spectrum as of $\mathcal{H}^\Omega_{V_\varphi}$, Theorem \ref{teo:maintheorem} applies to the first Dirichlet eigenvalue of $-\mathcal{L}$. In particular, when $A=B=Id$ and $b_0=0$ we have that $V_\varphi(x)=\frac{|x|^2}{4}-\frac N2$, which is the potential of the shifted harmonic oscillator. In this case $\mathcal{L}^\Omega$ reduces to the Dirichlet Ornstein-Uhlenbeck operator $\Delta^\Omega_\gamma$ and we recover the Brunn-Minkowski inequality for $\lambda_{1,\gamma}(\Omega)$ proved in \cites{CFLS, CLS}.
\end{example}

\begin{example}[Schr\"odinger operator with singular potential]

Let $N\ge 3$, $\Omega\subset\R^N$ a bounded convex set and $V(x):=\frac{1}{d(x,\partial\Omega)^2}$. Consider the associated quadratic form $\mathcal{Q}_{V}$ which is coercive thanks to the Poincaré inequality. By the boundary Hardy inequality for convex domains (see
\cite[Theorem 11]{MarMizPin98}), we have
$$
\int_\Omega
V|u|^2\,dx
\le
4\int_\Omega|\nabla u|^2\,dx,
\qquad u\in H_0^1(\Omega).
$$
it follows that $D(\mathcal{Q}_V)=H^1_0(\Omega)$, and the Rellich-Kondrachov Theorem ensures the compactness of the resolvent and hence the discreteness of the spectrum. Finally, it is an easy task to check that $V$ satisfies Assumptions \ref{As1}, \ref{As2}, \ref{As3}. Therefore, Theorem \ref{teo:maintheorem} holds true.
    
\end{example}

\subsection{Log-concavity of the ground state}
\label{sec:logconceig}

In this subsection we prove, as a consequence of Theorem \ref{teo:maintheorem}, the log-concavity of the first eigenfunction $\psi_{1,V}$ of $\mathcal{H}^\Omega_V$.  
As a consequence of Example \ref{example:Kolmogorov} we have that $U^{-1}_\varphi\psi_{1,V}=e^\varphi\psi_{1,V}$, where $U_\varphi$ denotes the isometry \eqref{eq:unitarytransform}, is the first Dirichlet eigenfunction for the Kolmogorov operator $-\operatorname{div}(A\nabla)+b\cdot\nabla$. In the particular case $A=Id$ and $b(x)=x$ the log-concavity of the ground state of $-\Delta_\gamma^\Omega$ has been proved in \cites{colesanti, CFLS, CLS}. Moreover, in the very recent paper \cite{Qin26} the author proves that $e^\varphi\psi_{1,V}$ enjoys a strong log-concavity property if $\Omega$ is bounded and convex. In all these papers, the set $\Omega$ is assumed to be bounded because of the unboundedness of the weight $e^\varphi$. Here, we prove the log-concavity of $\psi_{1,V}$ skipping the boundedness assumption on $\Omega$ and exploiting the ultracontractivity of the semigroup $(e^{-t\mathcal{H}^\Omega_V})_{t>0}$ which is not preserved by the isometry \eqref{eq:unitarytransform} and does not hold for Ornstein-Uhlenbeck type semigroups without the further assumption of \textit{intrinsic} ultracontractivity. We refer e.g. to the work \cite{DavSim84}.

Before to prove Theorem \ref{th:logconcdireig}, we recall the following result due to Prékopa, \cite[Theorem 6]{prekopa}.

\begin{theorem}
\label{th:prekopaint}
    Let $N_1,N_2\in\N$,  $A\subseteq\R^{N_1}$, $B\subseteq\R^{N_2}$ two convex sets and $f:A\times B\rightarrow (0,\infty)$ a log-concave function. Then, the function
    $$
    F(x):=\int_B f(x,y)dy,
    $$
    is log-concave in $A$.
\end{theorem}

\begin{proof}[Proof of Theorem \ref{th:logconcdireig}]
	Consider the following Cauchy-Dirichlet problem
	\begin{equation}
	\label{eq:firsteigenfun}
	\begin{cases}
		u_t+\mathcal{H}_Vu=0\quad\text{in}\quad\Omega\times(0,\infty) \\
		u(x,t)=0\quad\text{in}\quad\partial\Omega\times(0,\infty) \\
		u(x,0)=f(x)\quad\text{in}\quad\Omega
	\end{cases}
	\end{equation}
	for some $f\in L^2(\Omega)$, $f\ge 0$, $f\not\equiv 0$, $f$ log-concave in $\Omega$. The solution to \eqref{eq:firsteigenfun} is then given by
	$$
	u(x,t)=e^{-t\mathcal{H}^\Omega_V}f(x)=\int_\Omega p^\Omega_V(t,x,y)f(y)dy.
	$$
    By Theorem \ref{th:prekopaint}, since the function $\Omega\times\Omega\ni (x,y)\mapsto p^\Omega_V(t,x,y)f(y)$ is log-concave for every $t>0$, so does $u(x,t)$ for every $x\in\Omega$, $t>0$.
	Using the spectral representation of $p^\Omega_V$ it follows that
	$$
	e^{\lambda_{1,V}(\Omega)t}u(x,t)= (f,\psi_{1,V})_{L^2(\Omega)}\psi_{1,V}(x)+\sum_{k=2}^\infty e^{-(\lambda_{k,V}(\Omega)-\lambda_{1,V}(\Omega))t}(f,\psi_{k,V})_{L^2(\Omega)}\psi_{k,V}(x).
	$$
	Therefore, by Cauchy-Schwarz inequality we have that
	\begin{equation*}
	|I(x,t)|\le M(t)\left\|f\right\|_{L^2(\Omega)}\left(\sum_{k=2}^{\infty}e^{-\lambda_{k,V}(\Omega)}\right)^{1/2}\left(\sum_{k=2}^{\infty}e^{-\lambda_{k,V}(\Omega)}|\psi_{k,V}(x)|^2\right)^{1/2},
	\end{equation*}
	where for every $t>0$ and every $x\in \Omega$ we have set $I(x,t):=e^{\lambda_{1,V}(\Omega)t}u(x,t)-( f,\psi_{1,V})_{L^2(\Omega)}\psi_{1,V}(x)$ and $M(t):=\sup_{k\ge 2} e^{-(\lambda_{k,V}(\Omega)-\lambda_{1,V}(\Omega))t+\lambda_{k,V}(\Omega)}$. In particular, since the sequence $(\lambda_{k,V}(\Omega))_{k\in\N}$ is increasing and the ultracontractivity estimate \eqref{eq:ultracontractivity} holds true, for every $t\ge 1$ we have that 
    \begin{equation}
    \label{eq:logconcfe}
    \begin{split}
    |I(x,t)|&\le e^{-(\lambda_{2,V}(\Omega)-\lambda_{1,V}(\Omega)) t+\lambda_{2,V}(\Omega)}\left\|f\right\|_{L^2(\Omega)}\sqrt{Z_\Omega(1)}\sqrt{p^\Omega_V(1,x,x)} \\
    &\le C(\left\|f\right\|_{L^2(\Omega)},\left\|e^{-V}\right\|_{L^1(\Omega)}\omega_0, a_1)e^{-(\lambda_{2,V}(\Omega)-\lambda_{1,V}(\Omega)) t+\lambda_{2,V}(\Omega)}.
    \end{split}
    \end{equation}
    Passing to the supremum with respect to $x\in\Omega$ in \eqref{eq:logconcfe} and letting $t\to\infty$ we have that $\psi_{1,V}$ is log-concave in $\Omega$ since is the uniform limit of a sequence of log-concave functions. To conclude, $\psi_{1,V}$ is also strictly positive since for every $t>0$ and every $x\in\Omega$ we have that
    $$
    \psi_{1,V}(x)=e^{\lambda_{1,V}(\Omega)t}(e^{-t\mathcal{H}^\Omega_V}\psi_{1,V})(x).
    $$
    Then, the result plainly follows by Proposition \ref{prop:irrschdirsem}.
	\end{proof}

        \begin{corollary}
        Assume all the hypotheses of Theorem \ref{th:logconcdireig} are satisfied and also that $\Omega$ is bounded. Then, the first Dirichlet eigenfunction of the Kolmogorov operator $-\mathcal{L}=-\operatorname{div}(A\nabla)+b\cdot\nabla$ is log-concave in $\Omega$.
    \end{corollary}
\begin{proof}
    By Theorem \ref{th:kolesnikov} the semigroup generated by $\mathcal{L}$ preserves log-concavity. Then, we argue as in the proof of Theorem \ref{th:logconcdireig}.
\end{proof}

    We conclude the paper with the following Proposition \ref{prop:strictlc} which is an enhancement of Theorem \ref{th:logconcdireig}. We refer the reader to \cite[Section 4]{CFLS} to remark how the geometrical assumptions therein can be partially rephrased in our framework in terms of the potential $V$, and to \cites{Kor83,KorLew87}, where alternative approaches to prove the strong log-concavity of $\psi_{1,V}$ relying on the
convexity maximum principle and constant-rank techniques are used.

\begin{proposition}
\label{prop:strictlc}
Assume in Theorem \ref{th:logconcdireig} that
$\Omega$ is a bounded domain of class
$C^{2,\alpha}$, and suppose that
$$
V\in C^{0,\alpha}(\overline{\Omega})
\cap C^{2,\alpha}_{\mathrm{loc}}(\Omega)
$$
for some $\alpha\in(0,1)$. Assume moreover that $V$ is strongly
convex in $\Omega$, namely
$$
D^2V(x)\xi\cdot\xi>0
\quad\text{for every}\quad x\in\Omega,\,\xi\in\R^N\setminus\{0\}.
$$
Then, the first Dirichlet eigenfunction $\psi_{1,V}$ is strongly
log-concave in $\Omega$, or, equivalently
$$
D^2(-\log\psi_{1,V})(x)\xi\cdot\xi>0\quad\text{for every}\quad x\in\Omega,\,\xi\in\R^N\setminus\{0\}.
$$
\end{proposition}

\begin{proof}
Since
$$
V\in C^{0,\alpha}(\overline{\Omega})
\cap C^{2,\alpha}_{\mathrm{loc}}(\Omega),
$$
the global $W^{2,p}$-regularity and the interior and boundary
Schauder estimates
\cite[Theorem 9.15 and Theorems 6.17, 6.19]{GilTru01} yield
$$
\psi_{1,V}\in
C^{4,\alpha}_{\mathrm{loc}}(\Omega)
\cap C^{2,\alpha}(\overline{\Omega}).
$$
Since $\psi_{1,V}>0$ in $\Omega$, the function
$$
w:=-\log\psi_{1,V}
$$
belongs to $C^{4,\alpha}_{\mathrm{loc}}(\Omega)$ and satisfies
\begin{equation}
\label{eq:hamiltonjacobi}
\operatorname{div}(A\nabla w)
=
A\nabla w\cdot\nabla w+\lambda_{1,V}(\Omega)-V
\qquad\text{in }\Omega.
\end{equation}

By Theorem \ref{th:logconcdireig}, $w$ is convex. Hence, setting
$W:=D^2w$, we have
\begin{equation}
\label{eq:semidefpos}
W(x)\xi\cdot\xi\geq 0
\quad\text{for every}\quad x\in\Omega, \xi\in \R^N.
\end{equation}
Differentiating \eqref{eq:hamiltonjacobi} twice, and using that $A$
is constant, we obtain
\begin{equation}
\label{eq:Hessian_PDE}
\operatorname{div}(A\nabla W_{ij})
-2A\nabla w\cdot\nabla W_{ij}
-2(WAW)_{ij}+V_{ij}=0
\qquad\text{in }\Omega
\end{equation}
for every $i,j=1,\ldots,N$.

Suppose by contradiction that $W(x_0)$ is not positive definite at
some $x_0\in\Omega$. After an orthogonal change of coordinates, we
may assume that
$$
W(x_0)e_1=0.
$$
By \eqref{eq:semidefpos}, the function
$x\mapsto W_{11}(x)$ is nonnegative and attains its minimum at
$x_0$. Therefore, necessary conditions of minimality imply:
$$
\nabla W_{11}(x_0)=0,
\qquad
\operatorname{div}(A\nabla W_{11})(x_0)\geq0.
$$
Moreover,
$$
(WAW)_{11}(x_0)
=
\bigl(W(x_0)e_1\bigr)\cdot
A\bigl(W(x_0)e_1\bigr)
=0.
$$
Evaluating \eqref{eq:Hessian_PDE} at $x_0$ with $i=j=1$, we obtain
$$
\operatorname{div}(A\nabla W_{11})(x_0)
=
-V_{11}(x_0)<0,
$$
which contradicts the preceding minimum condition. Consequently,
$$
W(x)\xi\cdot\xi>0
\quad\text{for every}\quad x\in\Omega,\,\xi\in\R^N\setminus\{0\}.
$$
Thus $w$ is strongly convex, hence
$\psi_{1,V}$ is strongly log-concave in $\Omega$.
\end{proof}

\end{document}